\documentclass[12pt,reqno]{amsart}

\usepackage[l2tabu,orthodox]{nag}

\usepackage{hyperref}

\usepackage{enumerate}

\usepackage[swedish,english]{babel}
\usepackage[T1]{fontenc}
\usepackage[latin1]{inputenc}

\usepackage{a4wide}

\usepackage{amsmath,amssymb,amsthm}

\theoremstyle{plain}
\newtheorem{thm}{Theorem}[section]
\newtheorem{lem}[thm]{Lemma}
\newtheorem{prop}[thm]{Proposition}
\newtheorem{cor}[thm]{Corollary}
\theoremstyle{definition}

\newtheorem{exmp}[thm]{Example}

\theoremstyle{remark}
\newtheorem{rem}[thm]{Remark}

\theoremstyle{plain}
\newtheorem{mainthm}{Theorem}

\newcommand{\C}{\mathbb{C}}
\newcommand{\K}{\mathbb{K}}

\DeclareMathOperator{\End}{End}

\DeclareMathOperator{\GL}{GL}
\DeclareMathOperator{\im}{im}
\DeclareMathOperator{\rank}{rank}

\title[Invertibility of structured perturbations of singular matrices]{Invertibility of structured perturbations \\ of singular matrices over unital rings}
\date{\today}

\begin{document}

\author{Johan \"{O}inert}
\address{
Department of Mathematics and Natural Sciences,
Blekinge Institute of Technology,
SE-37179 Karlskrona, Sweden \\
and \\
Department of Engineering, 
University of Sk\"{o}vde, 
SE-54128 Sk\"{o}vde, Sweden}
\email{johan.oinert@bth.se}


\begin{abstract}
In this article, we study the invertibility of matrices of the form $A+EDF$, where $A$ and $D$ are square matrices over a unital, not necessarily commutative, ring and $A$ is singular. 
Under natural hypotheses on kernels and images, 
we prove that $A+EDF$ is invertible if and only if $D$ is invertible,
and obtain an explicit formula for the inverse of $A+EDF$ by
purely algebraic methods.
When the coefficient ring is stably finite, or more specifically a field, 
we obtain an invertibility criterion under weaker hypotheses. 
Our results extend and sharpen a theorem of Eriksson and Nordqvist for complex matrices.
\end{abstract}

\maketitle


\section{Introduction}\label{Sec:Introduction}

The problem of determining the inverse of a 
sum of matrices is classical (see e.g. \cite{HendersonSearle1981,Miller1981}).
For an invertible base matrix $A$,
the standard starting point 
is the 
Sherman--Morrison--Woodbury formula \cite{ShermanMorrison1950,Woodbury1950}
\begin{equation}\label{eq:ShermanMorrisonWoodbury}
		 (A+UCV)^{-1}
 =A^{-1}-A^{-1}U\bigl(C^{-1}+VA^{-1}U\bigr)^{-1}VA^{-1},
\end{equation}
whenever the indicated inverses exist. This identity, together with its rank-one special case, reduces the inversion of a modified matrix to that of the original matrix and a smaller auxiliary matrix. It is therefore fundamental in inverse-updating procedures (see e.g.
\cite{Hager1989} and the references therein).

Inverse-updating formulas occur throughout applied mathematics.
They are used in e.g. statistical estimation and covariance calculations \cite{HendersonSearle1981,Anderson1976}, optimization \cite{DennisMore1977}, network analysis \cite{AlsacStottTinney1983}, structural reanalysis \cite{AkgunGarcelonHaftka2001}, and numerical methods for partial differential equations \cite{ProskurowskiWidlund1976}. Their computational value is especially clear when a large matrix is altered by a perturbation of small rank. Singular base matrices also arise naturally, 
and in that setting Riedel \cite{Riedel1992} has studied rank-augmenting modifications and obtained a Sherman-Morrison-Woodbury-type identity for the Moore-Penrose inverse.
For rank-additive sums, Fill and Fishkind \cite{FillFishkind2000} obtained a formula expressing the Moore-Penrose inverse of the sum in terms of the generalized inverses of its summands.

When the base matrix $A$ is singular, the usual Sherman--Morrison--Woodbury formula \eqref{eq:ShermanMorrisonWoodbury} is not applicable. 
Motivated by this, Eriksson and Nordqvist 
\cite[Thm.~2.1]{ErikssonNordqvist2024}
considered complex matrices of the form
\begin{displaymath}
	A+EDF,
\end{displaymath}
where $A\in M_n(\mathbb C)$ is singular of rank $n-k>0$, $D\in M_k(\mathbb C)$ is invertible, $E\in M_{n\times k}(\mathbb C)$, and $F\in M_{k \times n}(\mathbb C)$.
Under the assumptions that the columns of $A$ together with those of $E$ span $\mathbb C^n$, and that the rows of $A$ together with those of $F$ span $\mathbb C^n$ (cf.~Remark~\ref{rem:Complex}), they obtain an explicit inverse of $A+EDF$. 
Their 
proof
uses finite-dimensional Hilbert space tools, in particular adjoints, orthogonal complements, and singular value decomposition.
They also give an expression for the inverse that does not require singular value decomposition.

The assumptions in their complex setting have natural module-theoretic interpretations, suggesting that rank,  
orthogonality, and singular value decomposition are not essential to the underlying inversion mechanism.
The appropriate hypotheses concern kernels, images, direct sum decompositions, and module isomorphisms.
In this article, we develop this point of view over an arbitrary unital ring.

Throughout this article, unless stated otherwise, let $R$ be a nonzero unital and associative, but not necessarily commutative, ring.
For $n,k>0$, regard $R^n$ and $R^k$ as free right $R$-modules. Matrices act on column vectors by left multiplication, and we identify a matrix with the corresponding homomorphism of right $R$-modules. Thus it is meaningful to speak of the kernel, image, injectivity, and surjectivity of a matrix. We write $M_{n\times k}(R)$ for the set of $n\times k$ matrices with entries from $R$, $M_n(R):=M_{n\times n}(R)$, $I_n$ for the identity matrix in $M_n(R)$, and $\GL_n(R)$ for the group of multiplication invertible matrices in $M_n(R)$.

Our first main result 
gives both an invertibility criterion and an explicit inversion formula for $A+EDF$ over an arbitrary unital ring.

\begin{mainthm}\label{thm:Main}
Let $n\geq k \geq 1$ be integers, and let
$A\in M_n(R)$,
$E\in M_{n\times k}(R)$,
$D\in M_k(R)$,
and $F\in M_{k\times n}(R).$
Suppose that the following three conditions hold:
\begin{enumerate}
\item[{\rm (P1)}] The restriction
$F|_{\ker(A)}:\ker(A) \to R^k$ 
is an isomorphism of right $R$-modules.

\item[{\rm (P2)}] $R^n=\im(A) \oplus \im(E).$

\item[{\rm (P3)}] $E$ is injective.
\end{enumerate}

\noindent 
Then $A+EDF \in \GL_n(R)$ if and only if $D \in \GL_k(R)$.
Moreover, if $D\in \GL_k(R)$, then $(A+EDF)^{-1}=G+XD^{-1}Y,$
where $G,X,Y$ are the matrices constructed in Lemma~\ref{lem:existenceGXY}.
\end{mainthm}

The proof is entirely algebraic. No determinant, adjoint, inner product, or singular value decomposition is used.

For arbitrary rings, the injectivity assumption on $E$ is essential: one-sided invertibility need not imply two-sided invertibility, and surjective endomorphisms of finitely generated free modules need not be injective. A natural setting in which the relevant one-sided obstruction disappears is that of stably finite rings. Recall that $R$ is \emph{stably finite} if $M_m(R)$ is directly finite for every $m\geq 1$, i.e. for all $m\geq 1$ and $X,Y \in M_m(R)$, $XY=I_m$ implies $YX=I_m$. 
In that setting, the hypotheses and conclusion admit a sharper formulation,
leading to our second main result.

\begin{mainthm}\label{thm:SecondMain}
Let $R$ be a stably finite ring,
let $n\geq k \geq 1$ be integers, and let
$A\in M_n(R)$,
$E\in M_{n\times k}(R)$,
$D\in M_k(R)$,
and $F\in M_{k\times n}(R).$
Suppose that the following condition holds:
\begin{enumerate}

\item[{\rm (S1)}] The restriction
$F|_{\ker(A)}:\ker(A) \to R^k$ 
is a surjective right $R$-module homomorphism.

\end{enumerate}

\noindent Then $A+EDF \in \GL_n(R)$
if and only if the following two assertions hold:
\begin{enumerate}[{\rm (i)}]
	\item $D \in \GL_k(R)$.
	\item $R^n=\im(A) + \im(E)$.
\end{enumerate}
Moreover, when it exists, the inverse of
$A+EDF$ is given by 
$(A+EDF)^{-1}=G+XD^{-1}Y,$
where $G,X,Y$ are the matrices constructed in Lemma~\ref{lem:existenceGXY}.
\end{mainthm}

Stable finiteness of $R$ is in fact necessary for the equivalence in Theorem~\ref{thm:SecondMain}
to hold for all matrix sizes and all matrices satisfying condition (S1) (see Remark~\ref{rem:SF}).

This article is organized as follows.
In Section~\ref{Sec:ProofA}, we construct the matrices $G$, $X$, and $Y$ and prove Theorem~\ref{thm:Main}.
Along the way, we record short proofs of the rectangular forms of Jacobson's lemma and Sylvester's determinant identity (see Lemma~\ref{lem:JacobsonRectangular}).
We also establish another determinant identity 
(see Proposition~\ref{prop:determinant}).
Section~\ref{Sec:Examples} provides examples showing that each of the hypotheses in Theorem~\ref{thm:Main} is needed over arbitrary rings.
In Section~\ref{Sec:ProofB}, we prove the stable finiteness result, Theorem~\ref{thm:SecondMain}.
Finally, in Section~\ref{Sec:Field} we specialize the theory to 
matrices with entries in a field
and 
prove Corollary~\ref{cor:Field} (see also Remark~\ref{rem:Complex}).
Thereby we simultaneously generalize, recover and sharpen 
the main result of Eriksson and Nordqvist \cite[Thm.~2.1]{ErikssonNordqvist2024}.

\section{Proof of Theorem~\ref{thm:Main}}\label{Sec:ProofA}

In this section, we establish our first main result, Theorem~\ref{thm:Main}.
We begin by deriving the existence of matrices $G,X,Y$ with certain prescribed properties.
The following lemma gives a module-theoretic construction of the relations appearing over $\C$ in \cite[Cor.~1]{ErikssonNordqvist2024}.

\begin{lem}\label{lem:existenceGXY}
Let
$A\in M_n(R)$,
$E\in M_{n\times k}(R)$,
and 
$F\in M_{k\times n}(R)$.	
Suppose that the following three conditions hold:
\begin{enumerate}
\item[{\rm (P1)}] The restriction
$F|_{\ker(A)}:\ker(A) \to R^k$ 
is an isomorphism of right $R$-modules.

\item[{\rm (P2)}] $R^n=\im(A) \oplus \im(E).$

\item[{\rm (P3)}] $E$ is injective.
\end{enumerate}
Then there exist matrices
$G\in M_n(R)$,
$X\in M_{n\times k}(R)$, and
$Y\in M_{k\times n}(R)$
such that
\begin{displaymath}
	\begin{array}{lllllll}
		AX=0, & \quad & GE=0, & \quad & YE=I_k, & \quad & AG+EY=I_n,\\
		&&&&&& \\
YA=0, & \quad & FG=0, & \quad & FX=I_k, & \quad & GA+XF=I_n.
	\end{array}
\end{displaymath}
\end{lem}

\begin{proof}
Using (P1), let $X:R^k\to\ker(A)\subseteq R^n$ be the inverse of $F|_{\ker(A)}$.
Then $AX=0$ and $FX=I_k$.

We claim that $R^n = \ker(F) \oplus \ker(A)$.
Clearly, each $v\in R^n$ can be written as $v = (v-XFv)+XFv$.
Note that $XFv \in \im(X)=\ker(A)$, and $v-XFv \in \ker(F)$ because $F(v-XFv)=Fv-FXFv=Fv-I_k Fv=0$.
Hence, $R^n = \ker(F) + \ker(A)$.
Take $w\in \ker(F) \cap \ker(A)$.
Then $XFw=w$ because $w\in \ker(A)$, and $Fw=0$ because $w\in \ker(F)$. Thus, $w=0$.
This proves the claim.

Using the above direct sum decomposition, it is easy to see that
the restriction
$A_0:=A|_{\ker(F)}:\ker(F)\to \im(A)$
is both injective and surjective, and hence an isomorphism.

By (P2) and (P3), 
every $v\in R^n$ can be written as
$v=a+Ec$ for unique elements $a\in \im(A)$ and $c\in R^k$.
Define $G : R^n \to R^n$ by
$G(a+Ec):=A_0^{-1}a.$
Then $FG=0$ and $GE=0$.
Also, define $Y : R^n \to R^k$ by $Y(a+Ec):=c$. Then $YA=0$ and $YE=I_k$.

The maps $AG$ and $EY$ are the complementary projections onto $\im(A)$ and $\im(E)$.
Indeed, for $v=a+Ec$ as above, we have $AGv=a$ and $EYv=Ec$.
Similarly, writing $v\in R^n$ as $v=u+w$ with $u\in\ker(F)$ and $w\in\ker(A)$, we obtain
$GAv=A_0^{-1}(Au)=u$ and $XFv=XFw=w.$
Thus,
$GA$ and $XF$ are the complementary projections onto $\ker(F)$ and $\ker(A)$.
Hence, $AG+EY=I_n$ and $GA+XF=I_n$.
%
\end{proof}

\begin{rem}
The matrices in Lemma~\ref{lem:existenceGXY}
satisfy $AGA=A$ and $GAG=G$, i.e. $G$ is a 
\emph{reflexive generalized inverse of $A$}.
Indeed,
$AGA=(I_n-EY)A=A-EYA=A-0=A$
and
$GAG=G(I_n-EY)=G-GEY=G-0=G$.
\end{rem}

For a unital ring, \emph{Jacobson's lemma} 
asserts that
for elements $a,b$, the element $1-ab$ is invertible if and only if $1-ba$ is invertible 
(see e.g. \cite[pp.~162-166]{Halmos1982} or  \cite[pp.~3-4]{Lam2003}).
Below we record a rectangular version of Jacobson's lemma,
and the \emph{Sylvester determinant theorem} for matrices over a unital commutative ring.

\begin{lem}\label{lem:JacobsonRectangular}
Let $n, k$ be positive integers, and
let $U \in M_{n \times k}(R)$ and $V \in M_{k \times n}(R)$.
Then the following two assertions hold:
\begin{enumerate}[{\rm (a)}]
	\item $I_n + UV \in \GL_n(R)$ if and only if $I_k + VU \in \GL_k(R)$.
	\item If $R$ is commutative, then $\det(I_n + UV)=\det(I_k + VU)$.
\end{enumerate}
\end{lem}

\begin{proof}
Consider the following three block matrices in $M_{n+k}(R)$:
\begin{displaymath}
A:=
\begin{bmatrix}
		I_n & 0\\
		V   & I_k
\end{bmatrix},
\quad
B:=
\begin{bmatrix}
		I_n & U\\
		-V   & I_k
\end{bmatrix},
\quad \text{ and } \quad
C:=
\begin{bmatrix}
		I_n & -U\\
		0   & I_k
\end{bmatrix}.
\end{displaymath}
After directly calculating products of these block matrices, we get that
\begin{equation}
\label{eq:JacobsonBlocks}
ABC =
\begin{bmatrix}
		I_n & 0 \\
		0 & I_k+VU
\end{bmatrix}
\quad
\text{ and }
\quad
CBA =
\begin{bmatrix}
		I_n+UV & 0 \\
		0 & I_k
\end{bmatrix}.
\end{equation}

(a) 
A routine check confirms that $A,C \in \GL_{n+k}(R)$.
Thus, $ABC \in \GL_{n+k}(R)$ if and only if $B\in \GL_{n+k}(R)$ if and only if $CBA \in \GL_{n+k}(R)$.
Hence, by \eqref{eq:JacobsonBlocks},
$I_k + VU \in \GL_k(R)$
if and only if
$I_n + UV \in \GL_n(R)$.

(b) Suppose that $R$ is commutative.
Using \eqref{eq:JacobsonBlocks}
and the fact that $\det(CBA)=\det(ABC)$,
we get that $\det(I_n + UV)=\det(I_k + VU)$.
\end{proof}

We are now ready to prove our first main result.

\begin{proof}[Proof of Theorem~\ref{thm:Main}]
Choose matrices
$G\in M_n(R)$,
$X\in M_{n\times k}(R)$,
and
$Y\in M_{k\times n}(R)$
as in Lemma~\ref{lem:existenceGXY}.
Write $T_D:=A+EDF$ and $T_1:=A+EF$.

First, suppose that $D \in \GL_k(R)$. 
Using that $AX=0$, $FG=0$, $FX=I_k$,
and $AG+EY=I_n$,
we have
\begin{align*}
T_D(G+XD^{-1}Y) &= AG+AXD^{-1}Y+EDFG+EDFXD^{-1}Y \\
&= AG + 0 + 0 + EI_kY = I_n.
\end{align*}
Similarly,
using that $GE=0$, $YA=0$, $YE=I_k$, 
and $GA+XF=I_n$,
we have
\begin{align*}
(G+XD^{-1}Y)T_D &= GA+GEDF+XD^{-1}YA+XD^{-1}YEDF \\
&= GA + 0 + 0 + XI_kF = I_n.
\end{align*}
This shows that $T_D^{-1}=G+XD^{-1}Y$.
Thus, we have established the ``if'' statement.

Now, suppose that $T_D \in \GL_n(R)$.
Note that, by the above, $T_1 = T_{I_k} \in \GL_n(R)$.
Using that $AX=0$ and $FX=I_k$, we have
$T_1X=(A+EF)X=AX+EFX=E$, and hence
\begin{equation*}
		T_D=T_1\bigl(I_n+X(D-I_k)F\bigr).
\end{equation*}
As $T_1 \in \GL_n(R)$, we have $I_n +X(D-I_k)F \in \GL_n(R)$.
By Lemma~\ref{lem:JacobsonRectangular}(a), 
with $U:=X$ and $V:=(D-I_k)F$, we have
$D=I_k+(D-I_k)I_k=I_k+(D-I_k)FX \in \GL_k(R)$.
This establishes the ``only if'' statement.
The form for the inverse of $T_D$ follows from the calculation in the beginning of the proof.
\end{proof}

We record the following result that generalizes \cite[Lem.~1]{ErikssonNordqvist2024}.

\begin{prop}\label{prop:determinant}
Let $R$ be a commutative ring, let $n\geq k \geq 1$ be integers, and let
$A\in M_n(R)$,
$E\in M_{n\times k}(R)$,
$D\in M_k(R)$,
and $F\in M_{k\times n}(R).$
Suppose that there is $X\in M_{n \times k}(R)$ such that
$AX=0$ and $FX=I_k$.
Then
\begin{equation*}
 \det(A+EDF)=\det(A+EF)\det(D).
\end{equation*}
\end{prop}

\begin{proof}
By assumption, $(A+EF)X=AX+EFX=0+E I_k = E$.
Hence, it follows that
$A+EDF=(A+EF)(I_n+X(D-I_k)F)$.
We now get that 
\begin{align*}
 \det(A+EDF)
 &=\det(A+EF)\det(I_n+X(D-I_k)F)\\
 &=\det(A+EF)\det(I_k+(D-I_k)FX)\\
 &=\det(A+EF)\det(I_k+(D-I_k) I_k) = \det(A+EF)\det(D),
\end{align*}
where we have used Lemma~\ref{lem:JacobsonRectangular}(b) in the second equality.
\end{proof}

\section{Examples}\label{Sec:Examples}

In this section, we present examples showing that the conclusion of Theorem~\ref{thm:Main} may fail
if any one of the conditions (P1), (P2) or (P3) is removed.

\begin{exmp}
Let $n:=2$ and $k:=1$. Define
$A:=
\begin{bmatrix}
1_R & 0\\
0 & 0
\end{bmatrix}$
and $D:=[1_R]\in \GL_1(R)$. 
Note that
$\ker(A)=
\left\{
\begin{bmatrix}
0\\
r
\end{bmatrix}: r\in R
\right\}$
and
$\im(A)=
\left\{
\begin{bmatrix} r\\ 0 \end{bmatrix}: r\in R
\right\}.$
We consider two cases:\\[-8pt]

(a)
Let
$E:=
\begin{bmatrix}
0\\
1_R
\end{bmatrix}$,
and
$F:=
\begin{bmatrix}
1_R & 0
\end{bmatrix}$.
Then $E\colon R\to R^2$ is injective. Moreover,
$\im(E)=
\left\{
\begin{bmatrix} 0\\ r \end{bmatrix}: r\in R
\right\},$
and hence
$R^2=\im(A)\oplus\im(E).$
Thus, conditions (P3) and (P2) in Theorem~\ref{thm:Main} hold. On the other hand,
$\ker(A) \subseteq \ker(F).$	
Therefore, $F|_{\ker(A)}\colon \ker(A)\to R$ is the zero map, and so condition
(P1) fails to hold.
Note that
$A+EDF = \begin{bmatrix}
1_R & 0\\
1_R & 0
\end{bmatrix}$
is not invertible.

(b)
Let $E:=
\begin{bmatrix}
1_R\\
0
\end{bmatrix}$,
and
$F:=
\begin{bmatrix}
0 & 1_R
\end{bmatrix}$.
Then $E\colon R\to R^2$ is injective. 
For every $r\in R$, we have 
$F
\begin{bmatrix}
0\\
r
\end{bmatrix}
=r,$	
and hence
$F|_{\ker(A)}\colon \ker(A)\to R$
is an isomorphism of right $R$-modules. 
Thus, conditions (P3) and (P1) hold.
On the other hand,
$\im(E)=
\left\{
\begin{bmatrix}
r\\
0
\end{bmatrix}: r\in R
\right\}.$
Hence, 
$\im(A)+\im(E)
=
\left\{
\begin{bmatrix}
r\\
0
\end{bmatrix}: r\in R
\right\}
\neq R^2.$	
Thus, condition (P2) fails to hold.
Note that
$A+EDF
=
\begin{bmatrix}
1_R & 1_R\\
0 & 0
\end{bmatrix}$
is not invertible.
\end{exmp}

\begin{exmp}
Let $\K$ be a field, let $V$ be a countably infinite-dimensional $\K$-vector
space with basis $\{e_0,e_1,e_2,\dots\}$ and set
$R:=\End_{\K}(V)$.
Define $S,L\in R$ by 
$S(e_i):=e_{i+1}$, $L(e_0):=0$, and $L(e_{i+1}):=e_i$.
Then
$LS=1_R$.
Thus, $L$ is right-invertible in $R$, so $LR=R$.
However, $L$ is not invertible
since $L$ is not injective as a $\K$-linear endomorphism of $V$.
Choose $P:=1_R - SL \in R$ and note that $P \neq 0$. 
Furthermore, $LP=L-LSL=0$.

Now choose $n:=2$ and $k:=1$.
Set
\begin{displaymath}
		A:=
\begin{bmatrix}
0 & 0\\
0 & 1_R
\end{bmatrix},
\qquad
E:=
\begin{bmatrix}
L\\
0
\end{bmatrix},
\qquad
D:=\begin{bmatrix} 1_R \end{bmatrix},
\qquad \text{ and }
\qquad
F:=\begin{bmatrix} 1_R & 0 \end{bmatrix}.
\end{displaymath}

Then
$\ker(A)=
\left\{
\begin{bmatrix}
r\\
0
\end{bmatrix}: r\in R
\right\}
$
and the restriction
$F|_{\ker(A)}:\ker(A)\to R$
is an isomorphism of right $R$-modules.
Note that
$\im(E)=
\left\{
\begin{bmatrix}
r \\
0
\end{bmatrix}: r \in R
\right\}
$
while
$\im(A)=
\left\{
\begin{bmatrix}
0\\
r
\end{bmatrix}: r\in R
\right\}$.
Hence,
$R^2=\im(A)\oplus\im(E)$.
We conclude that (P1) and (P2) in Theorem~\ref{thm:Main} hold, but (P3) does not hold because $EP=0$.
We also note that
$A+EDF =
\begin{bmatrix}
L & 0\\
0 & 1_R
\end{bmatrix}$
is not invertible, because $\begin{bmatrix}
P\\
0
\end{bmatrix} 
\in \ker(A+EDF)$.
\end{exmp}

\section{Proof of Theorem~\ref{thm:SecondMain}}\label{Sec:ProofB}

In this section, we focus on the case where $R$ is stably finite and prove our second main result, Theorem~\ref{thm:SecondMain}.

\begin{lem}\label{lem:dependencies}
Let $R$ be a stably finite ring,
let $n\geq k \geq 1$ be integers, and let
$A\in M_n(R)$,
$E\in M_{n\times k}(R)$,
and $F\in M_{k\times n}(R).$
Suppose that the restriction
$F|_{\ker(A)}:\ker(A) \to R^k$ 
is a surjective right $R$-module homomorphism,
and that $R^n=\im(A) + \im(E).$
Then the following three conditions hold:
\begin{enumerate}
\item[{\rm (P1)}] The restriction
$F|_{\ker(A)}:\ker(A) \to R^k$ 
is an isomorphism of right $R$-modules.

\item[{\rm (P2)}] $R^n=\im(A) \oplus \im(E).$

\item[{\rm (P3)}] $E$ is injective.
\end{enumerate}
\end{lem}

\begin{proof}
Since $R^k$ is free, the surjection $F|_{\ker(A)} : \ker(A) \to R^k$ admits a section $R^k \to \ker(A)$ which is a right $R$-module homomorphism.
Regard this section as a matrix $X\in M_{n\times k}(R)$, so that $AX=0$ and $FX=I_k.$

Set $T:=A+EF$. Then
$TX=E$ and $T(I_n-XF)=A.$
Hence, it follows that
$\im(T) =\im(A)+\im(E) =R^n.$
Thus, $T$ is surjective and hence there is some $S\in M_n(R)$ such that $TS=I_n$.
By stable finiteness of $R$, $T$ is invertible in $M_n(R)$.

If $v\in\ker(A)\cap\ker(F)$, then $Tv=0$, and hence $v=0$. Thus, $F|_{\ker(A)}$ is injective and therefore an isomorphism. This shows that (P1) holds.
Furthermore, $FX=I_k$ implies that $X$ is injective. Hence, $E=TX$ is injective, showing that (P3) holds.

Finally, suppose that $Au=Ec$ for some $u\in R^n$ and $c\in R^k$.
Since $A=T(I_n-XF)$, $E=TX$, and $T$ is invertible, we get that
$(I_n-XF)u=Xc$.
Applying $F$ and using $FX=I_k$ gives
$c=FXc=F(I_n-XF)u=0.$
Hence, $\im(A) \cap \im(E)= \{0\}$.
Together with our assumption, this shows that $R^n = \im(A) \oplus \im(E)$ and hence condition (P2) holds. 
\end{proof}

We are now ready to prove our second main result.

\begin{proof}[Proof of Theorem~\ref{thm:SecondMain}]
We first show the ``only if'' statement.
Suppose that $A+EDF \in \GL_n(R)$.
Then
$R^n=\im(A+EDF) \subseteq \im(A)+\im(E) \subseteq R^n,$
so (ii) holds.
Lemma~\ref{lem:dependencies} therefore gives (P1)--(P3), and Theorem~\ref{thm:Main} yields $D\in\GL_k(R)$, so (i) holds.

Now we show the ``if'' statement.
Suppose that (i) and (ii) hold. By Lemma~\ref{lem:dependencies}, conditions (P1)--(P3) hold.
Theorem~\ref{thm:Main} now gives $A+EDF \in\GL_n(R)$.
In both cases Theorem~\ref{thm:Main} yields the stated inversion formula.
\end{proof}

\begin{rem}\label{rem:SF}
Stable finiteness of $R$ is necessary in order for the equivalence in Theorem~\ref{thm:SecondMain} to hold 
for all matrix sizes and all matrices satisfying (S1).
Indeed, take any $n>0$ and suppose
that $P,Q \in M_n(R)$ satisfy $PQ=I_n$.
Set $k:=n$, $A:=0$, $E:=P$, $D:=I_n$, and $F:=I_n$.
Then (S1) holds and $R^n = \im(E)$, because $P$ is surjective.
Assuming that the equivalence of Theorem~\ref{thm:SecondMain} holds,
$P=A+EDF$ is invertible.
Using that $n$ was arbitrarily chosen we conclude that $R$ is stably finite.
\end{rem}

\section{Field coefficients}\label{Sec:Field}

In this section, we generalize and sharpen the main result of Eriksson and Nordqvist \cite[Thm.~2.1]{ErikssonNordqvist2024}.

\begin{cor}\label{cor:Field}
Let $\K$ be a field,
let $n\geq k \geq 1$ be integers, and let
$A\in M_n(\K)$,
$E\in M_{n\times k}(\K)$,
$D\in M_k(\K)$,
and $F\in M_{k\times n}(\K).$
Suppose that the following condition holds:
\begin{enumerate}
	\item[{\rm (C1)}] The rank of $A$ is $n-k$.
\end{enumerate}

\noindent Then $A+EDF \in \GL_n(\K)$
if and only if the following three assertions hold:
\begin{enumerate}[{\rm (i)}]
	\item $D \in \GL_k(\K)$.
	\item $\K^n$ is spanned by the column vectors of 
	$\begin{bmatrix}
	A & E 
	\end{bmatrix}.$
	\item $\K^n$ is spanned by the row vectors of 
	$\begin{bmatrix}
	A \\
	F
	\end{bmatrix}.$
\end{enumerate}
Moreover, when it exists, the inverse of
$A+EDF$ is given by 
$(A+EDF)^{-1}=G+XD^{-1}Y,$
where $G,X,Y$ are the matrices constructed in Lemma~\ref{lem:existenceGXY}.
\end{cor}

\begin{proof}
We first show the ``only if'' statement.
Suppose that $A+EDF \in GL_n(\K)$.
The block factorizations
$A+EDF=
\begin{bmatrix} A & E \end{bmatrix}
\begin{bmatrix} I_n \\ DF \end{bmatrix}
$
and
$A+EDF=
\begin{bmatrix} I_n & ED \end{bmatrix}
\begin{bmatrix} A \\ F \end{bmatrix}
$
give, via e.g. \cite[Appendix 1]{Strang2023}, that
$n=\rank(A+EDF)
\leq \rank\left(\begin{bmatrix} A & E \end{bmatrix}\right)
\leq n$	
and
$n=\rank(A+EDF)
\leq \rank\left(\begin{bmatrix} A \\ F \end{bmatrix}\right)
\leq n,$
respectively.
Thus, (ii) and (iii) both hold.
Furthermore, by (C1) we have
$n=\rank(A+EDF)
\leq \rank(A)+\rank(EDF)
\leq n-k+\rank(D).$
It follows that $\rank(D)=k$, showing that (i) holds.

Now we show the ``if'' statement.
Suppose that (i), (ii) and (iii) hold.
We aim to verify the 
hypotheses of Theorem~\ref{thm:SecondMain}.
By (i), $D\in \GL_k(\K)$.
By (ii), we have $\K^n = \im(A) + \im(E)$.
By (iii), $\ker\left(\begin{bmatrix} A \\ F \end{bmatrix}\right) = \{0\}$,
and hence
$\ker (A)\cap\ker (F)
=
\ker\left(\begin{bmatrix} A \\ F \end{bmatrix}\right)
=\{0\}.$
Therefore, the map
$F|_{\ker (A)}:\ker(A) \to \K^k$
is injective. By the rank--nullity theorem and (C1),
$\dim(\ker (A))=n-\rank(A)=n-(n-k)=k.$
Hence, $F|_{\ker (A)}:\ker(A) \to \K^k$ is a vector space isomorphism.

Finally, $\K$ is a field and therefore stably finite. The desired conclusion, including the
inversion formula, now follows from Theorem~\ref{thm:SecondMain}.
\end{proof}

\begin{rem}\label{rem:Complex}
For a complex matrix $M$, we write $M^*$ for its conjugate transpose.

Suppose that $A\in M_n(\C)$ has rank $n-k$, $D\in M_k(\C)$ is invertible, $E\in M_{n\times k}(\C)$, and $F\in M_{k\times n}(\C)$.
Furthermore, suppose that
the columns of $A$ together with those of $E$ span $\C^n$, 
and that the columns of $A^*$ together with those of $F^*$ span $\C^n$.

Condition (C1) and assertions (i)--(ii) of Corollary~\ref{cor:Field} are plainly satisfied.
Using that \linebreak
$\rank\left(
\begin{bmatrix}
A^* & F^*
\end{bmatrix}
\right) = \rank
\left(
\begin{bmatrix}
A \\
F
\end{bmatrix}\right)$
we get that (iii) of Corollary~\ref{cor:Field} also holds.
By Corollary~\ref{cor:Field}, $A+EDF$ is invertible and its inverse has the desired shape. 
Hence, we recover the result from \cite[Thm.~2.1]{ErikssonNordqvist2024}.
\end{rem}


\begin{thebibliography}{99}

\bibitem{AkgunGarcelonHaftka2001}
M. A. Akg\"un, J. H. Garcelon, and R. T. Haftka,
Fast exact linear and non-linear structural reanalysis and the
Sherman--Morrison--Woodbury formulas,
\emph{Internat. J. Numer. Methods Engrg.} \textbf{50} (2001), no. 7, 1587--1606.

\bibitem{AlsacStottTinney1983}
O. Alsac, B. Stott, and W. F. Tinney,
Sparsity-oriented compensation methods for modified network solutions,
\emph{IEEE Trans. Power Appar. Syst.} \textbf{PAS-102} (1983), no. 5, 1050--1060.

\bibitem{Anderson1976}
O. D. Anderson,
On the inverse of the autocovariance matrix for a general moving average process,
\emph{Biometrika} \textbf{63} (1976), no. 2, 391--394.

\bibitem{DennisMore1977}
J. E. Dennis, Jr., and J. J. Mor\'e,
Quasi-Newton methods, motivation and theory,
\emph{SIAM Rev.} \textbf{19} (1977), no. 1, 46--89.

\bibitem{ErikssonNordqvist2024}
S. Eriksson and J. Nordqvist, 
Inverting the sum of two singular matrices,
\emph{Results Appl. Math.} \textbf{22} (2024), Paper No. 100463, 5 pp.

\bibitem{FillFishkind2000}
J. A. Fill and D. E. Fishkind,
The Moore-Penrose generalized inverse for sums of matrices,
\emph{SIAM J. Matrix Anal. Appl.} \textbf{21} (2000), no. 2, 629--635.

\bibitem{Hager1989}
W. W. Hager, 
Updating the inverse of a matrix,
\emph{SIAM Rev.} \textbf{31} (1989), no. 2, 221--239.

\bibitem{Halmos1982}
P. R. Halmos, 
Does mathematics have elements?,
\emph{Bull. Austral. Math. Soc.} \textbf{25} (1982), no. 2, 161--175.

\bibitem{HendersonSearle1981}
H. V. Henderson and S. R. Searle, 
On deriving the inverse of a sum of matrices,
\emph{SIAM Rev.} \textbf{23} (1981), no. 1, 53--60.

\bibitem{Lam2003}
T. Y. Lam, 
\emph{Exercises in classical ring theory},
Second edition, Probl. Books in Math.,
Springer-Verlag, New York (2003).

\bibitem{Miller1981}
K. S. Miller, 
On the inverse of the sum of matrices,
\emph{Math. Mag.} \textbf{54} (1981), no. 2, 67--72.

\bibitem{ProskurowskiWidlund1976}
W. Proskurowski and O. Widlund,
On the numerical solution of Helmholtz's equation by the capacitance matrix method,
\emph{Math. Comp.} \textbf{30} (1976), 433--468.

\bibitem{Riedel1992}
K. S. Riedel,
A Sherman--Morrison--Woodbury identity for rank augmenting matrices with application to centering,
\emph{SIAM J. Matrix Anal. Appl.} \textbf{13} (1992), no. 2, 659--662.

\bibitem{ShermanMorrison1950}
J. Sherman and W. J. Morrison, 
Adjustment of an inverse matrix corresponding to a change in one element of a given matrix,
\emph{Ann. Math. Statistics} \textbf{21} (1950), 124--127.

\bibitem{Strang2023}
G. Strang,
\emph{Introduction to Linear Algebra},
6th ed., Wellesley-Cambridge Press (2023).

\bibitem{Woodbury1950}
M. A. Woodbury,
\emph{Inverting Modified Matrices}, Memorandum Report 42,
Statistical Research Group, Princeton University, Princeton, NJ (1950).

\end{thebibliography}
\end{document}